\documentclass[12pt]{amsart}
\usepackage{enumerate}
\usepackage{helvet,courier}
\usepackage{amsmath}
\usepackage{amssymb}

\DeclareFontFamily{U}{mathx}{\hyphenchar\font45}
\DeclareFontShape{U}{mathx}{m}{n}{
      <5> <6> <7> <8> <9> <10>
      <10.95> <12> <14.4> <17.28> <20.74> <24.88>
      mathx10
      }{}
\DeclareSymbolFont{mathx}{U}{mathx}{m}{n}
\DeclareFontSubstitution{U}{mathx}{m}{n}
\DeclareMathAccent{\widecheck}{0}{mathx}{"71}
\DeclareMathAccent{\wideparen}{0}{mathx}{"75}

\theoremstyle{plain}
\newtheorem{thm}{Theorem}
\newtheorem{lm}[thm]{Lemma}

\theoremstyle{remark}

\theoremstyle{definition}

\newcommand{\R}{\mathbb R}
\newcommand{\Z}{\mathbb Z}
\newcommand{\N}{\mathbb N}

\newcommand{\bnu}{\begin{enumerate}}
\newcommand{\enu}{\end{enumerate}}

\newcommand{\al}{\alpha}

\newcommand{\nf}{\infty}

\newcommand{\wh}{\widehat}

\newcommand{\mc}{\mathcal}

\newcounter{question}

\newcommand{\bpf}{\begin{proof}}

\newcommand{\epf}{\end{proof}}

\allowdisplaybreaks

\makeindex  

\begin{document}

\author{Lenka Slav\'ikov\'a}

\address{Department of Mathematics, University of Missouri, Columbia MO 65211, USA}
\email{slavikoval@missouri.edu}


\title[The H\"ormander multiplier theorem in a limiting case]{On the failure of the H\"ormander multiplier theorem in a limiting case}
\date{}

\begin{abstract}
We discuss the H\"ormander multiplier theorem for $L^p$ boundedness of Fourier multipliers in which the multiplier belongs to a fractional Sobolev space with smoothness $s$. We show that this theorem does not hold in the limiting case $|1/p-1/2|=s/n$.
\end{abstract}

\maketitle

\section{Introduction}

Let $m$ be a bounded function on $\R^n$. We define the associated linear operator
$$
T_m(f)(x) = \int_{\R^n} \wh{f}(\xi) m(\xi) e^{2\pi i x\cdot \xi}d\xi, \quad x\in \R^n,
$$
where $f$ is a Schwartz function on $\R^n$ and
$\wh{f}(\xi) = \int_{\R^n} f(x)   e^{-2\pi i x\cdot \xi}dx$ is the Fourier transform of $f$. The problem of characterizing the class of functions $m$ for which the operator $T_m$ admits a bounded extension from $L^p(\R^n)$ to itself for a given $p\in (1,\nf)$ is one of the principal questions in harmonic analysis. We say that $m$ is an $L^p$ Fourier multiplier if the above mentioned property is satisfied. 
While it is a straightforward consequence of Plancherel's identity that all bounded functions are $L^2$ Fourier multipliers, the structure of the set of $L^p$ Fourier multipliers for $p\neq 2$ turns out to be significantly more complicated. 

A classical theorem of Mikhlin~\cite{Mikhlin} asserts that if the condition
\begin{equation}\label{E:mikhlin}
|\partial^\alpha m(\xi)|\leq C_\alpha |\xi|^{-| \alpha|}, \qquad \xi\neq 0, 
\end{equation}
is satisfied for all { multi-indices} $\al$ with size $|\al | \le [n/2]+1$, then $T_m$ admits a bounded extension from
$L^p(\R^n)$ to itself for all $1<p<\nf$. A subsequent result by H\"ormander~\cite{Hoe} showed that the pointwise estimate~\eqref{E:mikhlin} can be replaced by a weaker Sobolev-type condition:
\begin{equation}\label{E:hormander}
\sup_{R>0} R^{-n+2|\alpha|} \int_{\{\xi \in \R^n: ~R<|\xi|<2R\}} |\partial^{\alpha} m(\xi)|^2\,d\xi <\infty.
\end{equation}
Although theorems of Mikhlin and H\"ormander admit a variety of applications, their substantial limitation stems from the fact that they can only be applied to functions which are $L^p$ Fourier multipliers for all values of $p \in (1,\infty)$.
One can overcome this difficulty using an interpolation argument as in Calder\'on and Torchinsky~\cite{CT} or Connett and Schwartz~\cite{CS1, CS}; the conclusion is, roughly speaking, that \textit{the closer $p$ is to $2$, the fewer derivatives are needed in conditions~\eqref{E:mikhlin} or~\eqref{E:hormander}}. 

To be able to formulate things precisely, let us now recall the notion of fractional Sobolev spaces. For $s>0$ we denote by $(I-\Delta)^{\frac{s}{2}}$ the operator given on the Fourier transform side by multiplication by $(1+4\pi^2 |\xi|^2)^{\frac{s}{2}}$. If $1<r<\infty$ then the norm in the fractional Sobolev space $L^r_s$ is defined by
$$
\|f\|_{L^r_s}=\|(I-\Delta)^{\frac{s}{2}}f\|_{L^r}.
$$

The version of the Mikhlin-H\"ormander multiplier theorem due to Calder\'on and Torchinsky~\cite[Theorem 4.7]{CT} says that inequality
\begin{equation}\label{E:quantitative_hormander}
\|T_mf\|_{L^p} \leq C \sup_{D\in \Z} \|\phi(\xi) m(2^D \xi)\|_{L^r_s} \|f\|_{L^p}
\end{equation}
holds provided that 
\begin{equation}\label{E:indices}
\left|\frac{1}{p}-\frac{1}{2}\right|=\frac{1}{r}<\frac{s}{n}.
\end{equation}
Here, $\phi$ stands for a smooth function on $\R^n$ supported in the set $\{\xi\in \R^n: 1/2<|\xi|<2\}$ and satisfying $\sum_{D\in \Z} \phi(2^D \cdot)=1$. Additionally, it was pointed out in~\cite{GraHeHonNg} that the equality $|1/p-1/2|=1/r$ is not essential for~\eqref{E:quantitative_hormander} to be true, and~\eqref{E:indices} can thus be replaced by the couple of inequalities
\begin{equation}\label{E:indices_new}
\left|\frac{1}{p}-\frac{1}{2}\right|<\frac{s}{n}, \quad \frac{1}{r}<\frac{s}{n}.
\end{equation}

Let us notice that the latter inequality in~\eqref{E:indices_new} is dictated by the embedding of $L^r_s$ into the space of essentially bounded functions. Related to this we also mention
that the Sobolev-type condition in~\eqref{E:quantitative_hormander} can be further weakened by replacing the Sobolev space $L^r_s$ with $r>n/s$ by the Sobolev space with smoothness $s$ built upon the Lorentz space $L^{n/s,1}$, see~\cite{GS}.

Let us now discuss the sharpness of the first condition in~\eqref{E:indices_new}. It is well known that if inequality~\eqref{E:quantitative_hormander} holds, then we necessarily have $|1/p-1/2| \leq s/n$, see~\cite{hirschman}, \cite{W}, \cite{Miy}, \cite{MT} and~\cite{GraHeHonNg}. On the critical line $|1/p-1/2|=s/n$ there are positive endpoint results by Seeger~\cite{Seeger1, Seeger2, Seeger3}. In particular, it is shown in~\cite{Seeger3} that inequality~\eqref{E:quantitative_hormander} holds when $|1/p-1/2|=s/n$ and $r>n/s$ if the Sobolev space $L^r_s$ is replaced by the Besov space $B^s_{1,r}$, defined by
$$
\|f\|_{B^s_{1,r}}=\sum_{k=0}^\infty 2^{ks} \|(\varphi_k \wh{f}) \widecheck{~}\|_{L^r}.
$$
Here, $\varphi_0$ stands for a Schwartz function on $\R^n$ such that $\varphi_0(x)=1$ if $|x|\leq 1$ and $\varphi_0(x)=0$ if $|x|\geq 3/2$, and $\varphi_k(x)=\varphi_0(2^{-k}x)-\varphi_0(2^{1-k}x)$ for $k\in \N$. We recall that $B^s_{1,r}$ is embedded into $L^r_s$, thanks to the equivalence
$$
\|f\|_{L^r_s} \approx \left\|\left(\sum_{k=0}^\infty 2^{2ks} |(\varphi_k \wh{f}) \widecheck{~}|^2\right)^{\frac{1}{2}}\right\|_{L^r}
$$
and to embeddings between sequence spaces. 

In this note we show that H\"ormander's condition involving the Sobolev space $L^r_s$ fails to guarantee $L^p$ boundedness of $T_m$ in the limiting case 
$|1/p-1/2|=s/n$.
Our result has the following form.

\begin{thm}\label{T:main_theorem}
Let $1<p<\infty$, $p\neq 2$, and let $s>0$ be such that 
\begin{equation}\label{E:assumption_s}
\left|\frac{1}{p}-\frac{1}{2}\right|=\frac{s}{n}.
\end{equation}
Assume that $\phi$ is a smooth function on $\R^n$ supported in the set $\{\xi\in \R^n: 1/2<|\xi|<2\}$ and $r>1$ is a real number.
Then there is no constant $C$ such that inequality
\begin{equation}\label{E:not_true}
\|T_mf\|_{L^p} \leq C \sup_{D\in \Z} \|\phi(\xi) m(2^D \xi)\|_{L^r_s} \|f\|_{L^p}
\end{equation}
holds for all $m$ and $f$.
\end{thm}

We would like to point out that Theorem~\ref{T:main_theorem} is for simplicity formulated in terms of the spaces $L^r_s$, but its proof can be easily adapted to show the failure of H\"ormander's theorem on the critical line $|1/p-1/2|=s/n$ for a much larger family of Sobolev-type spaces. In particular, any Lorentz-Sobolev space $L^{r_1,r_2}_s$, defined as
$$
\|f\|_{L^{r_1,r_2}_s}=\|(I-\Delta)^{\frac{s}{2}} f\|_{L^{r_1,r_2}},
$$
where
$$
\|g\|_{L^{r_1,r_2}}=
\begin{cases}
\left(\int_0^\infty t^{\frac{r_2}{r_1}-1} (g^*(t))^{r_2}\,dt\right)^{\frac{1}{r_2}} & \textup{if } 1<r_1<\infty \textup{ and } 1\leq r_2<\infty\\
\sup_{t>0} t^{\frac{1}{r_1}} g^*(t) & \textup{if } 1<r_1<\infty \textup{ and } r_2=\infty,
\end{cases}
$$
can be included in our discussion. Here, 
$$
g^*(t)=\inf\{\lambda >0:~ |\{x\in \R^n: |g(x)|>\lambda\}|\leq t\}, \quad t>0,
$$
stands for the nonincreasing rearrangement of $g$.
Our results thus provide a negative answer to the open problem A.2 raised in Appendix A of the recent paper~\cite{SW}.

The proof of Theorem~\ref{T:main_theorem} uses the randomization technique in the spirit of~\cite[chapter 4]{Wolff}, which has been further developed in~\cite{GraHeHonNg} and~\cite{GHS}. 

\section{Proof of Theorem~\ref{T:main_theorem}}

Let $s>0$ and let $\Psi$ be a non-identically vanishing Schwartz function on $\R^n$ supported in the set $\{\xi \in \R^n:|\xi|< 1/2\}$. Then for any fixed integer $K$ and for any $t\in [0,1]$ we define
$$
m_t(\xi)=\sum_{N=1}^K c_N \sum_{k\in \N^n:~ N2^N <|k|<(N+1/2)2^N} a_{N,k}(t) \Psi(2^N\xi-k),
$$
where $a_{N,k}(t)$ denotes the sequence of Rademacher functions indexed by the elements of the countable set $\N \times \N^n$, and $c_N=2^{-Ns} N^{-s}$. 

\begin{lm}\label{L:lemma}
Let $1<r<\infty$ and let $\phi$ be as in Theorem~\ref{T:main_theorem}. Then
$$
\sup_{D\in \Z} \|\phi(\xi) m_t(2^D\xi)\|_{L^r_s}\leq C,
$$
with $C$ independent of $t$ and $K$. 
\end{lm}

\begin{proof}
Let us fix $D\in \Z$. We denote
$$
A_D=\left\{\xi \in \R^n: \frac{1}{2}-\frac{1}{4\cdot 2^D}<|\xi|<2+\frac{3}{4\cdot 2^D}\right\}. 
$$
Using the version of the Kato-Ponce inequality from~\cite{GO}, we get
\begin{align}\label{E:kato-ponce}
&\|\phi(\xi)m_t(2^D\xi)\|_{L^r_s}\\
\nonumber
&=\|(I-\Delta)^{\frac{s}{2}} [\phi(\xi) m_t(2^D\xi)]\|_{L^r}\\
\nonumber
&=\|(I-\Delta)^{\frac{s}{2}} [\phi(\xi) \chi_{A_D}(\xi)m_t(2^D\xi)]\|_{L^r}\\
\nonumber
&\lesssim \|(I-\Delta)^{\frac{s}{2}} [\phi(\xi)]\|_{L^{\infty}} \|\chi_{A_D}(\xi)m_t(2^D\xi)\|_{L^r}\\
\nonumber
&+\|\phi(\xi)\|_{L^{\infty}} \|(I-\Delta)^{\frac{s}{2}} [\chi_{A_D}(\xi)m_t(2^D\xi)]\|_{L^r}\\
\nonumber
&\lesssim \|\chi_{A_D}(\xi)m_t(2^D\xi)\|_{L^r} + \|(-\Delta)^{\frac{s}{2}} [\chi_{A_D}(\xi)m_t(2^D\xi)]\|_{L^r} ,
\end{align}
since $\phi$ is a Schwartz function.

For any integer $N\leq K$ and for any $t\in [0,1]$ we denote 
$$
F_{N,t}(\xi)=c_N \sum_{k\in \N^n:~ N2^N <|k|<(N+1/2)2^N} a_{N,k}(t) \Psi(2^N\xi-k)
$$
and observe that $F_{N,t}$ is supported in the set $\{\xi \in \R^n: N-1/4<|\xi|<N+3/4\}$. Indeed, let $\xi\in \R^n$ belong to the support of $F_{N,t}$. Then there is $k\in \N^n$ satisfying $N2^N <|k|<(N+1/2)2^N$ such that $|2^N \xi -k| \leq 1/2$. Thus,
$$
|\xi|\leq \left|\xi-\frac{k}{2^N}\right|+\left|\frac{k}{2^N}\right| <\frac{1}{2^{N+1}}+N+\frac{1}{2}
\leq N+\frac{3}{4}.
$$
Conversely,
$$
|\xi|\geq \left|\frac{k}{2^N}\right|-\left|\frac{k}{2^N}-\xi\right| > N-\frac{1}{2^{N+1}} \geq N-\frac{1}{4},
$$
which justifies the claim. 

We now observe that if $\xi\in A_D$, then $2^{D-1}-1/4<|2^D\xi|<2^{D+1}+3/4$, and therefore $2^D\xi$ can only belong to the support of $F_{N,t}$ for $N$ satisfying $\max(1,2^{D-1}) \leq N \leq \min(K,2^{D+1})$. (We will thus assume that $\max(1,2^{D-1}) \leq \min(K,2^{D+1})$ in what follows; in particular, $D\geq -1$.) On the other hand, if $2^D\xi$ belongs to the support of $F_{N,t}$ for some $N$ with $\max(1,2^{D-1}) \leq N \leq \min(K,2^{D+1})$, then $2^{D-1}-1/4<|2^D\xi|<2^{D+1}+3/4$, and so $\chi_{A_D}(2^D\xi)=1$. These observations yield
\begin{align}\label{E:decomposition}
&\chi_{A_D}(\xi) m_t(2^D\xi)
=\sum_{N=\max(1,2^{D-1})}^{\min(K,2^{D+1})}F_{N,t}(2^D\xi)\\
&=\sum_{N=\max(1,2^{D-1})}^{\min(K,2^{D+1})} \sum_{k\in \N^n:~ N2^N <|k|<(N+1/2)2^N} c_N a_{N,k}(t) \Psi(2^{N+D}\xi-k). \nonumber
\end{align}

Thanks to the support properties of $\Psi$, the functions $\Psi(2^{N+D}\xi -k)$ have pairwise disjoint supports in $N$ and $k$ (for the fixed $D$). Combining this with the fact that $|c_N|\leq 1$, we deduce that the function $|\chi_{A_D}(\xi) m_t(2^D\xi)|$ is pointwise bounded by $\sup_{\xi \in \R^n} |\Psi(\xi)|$, and therefore
\begin{equation}\label{E:L^r_estimate}
\|\chi_{A_D}(\xi) m_t(2^D\xi)\|_{L^r}\leq C(n,r,\Psi).
\end{equation}

Further, we denote $\Phi=(-\Delta)^{\frac{s}{2}} \Psi$ and observe that
\begin{align*}
&(-\Delta)^{\frac{s}{2}} [\chi_{A_D}(\cdot)m_t(2^D\cdot)](\xi)\\
&=\sum_{N=\max(1,2^{D-1})}^{\min(K,2^{D+1})} 2^{(N+D)s} \sum_{k\in \N^n:~ N2^N <|k|<(N+1/2)2^N} c_N a_{N,k}(t) \Phi(2^{N+D}\xi-k)\\
&=\sum_{N=\max(1,2^{D-1})}^{\min(K,2^{D+1})} 2^{Ds} N^{-s} \sum_{k\in \N^n:~ N2^N <|k|<(N+1/2)2^N} a_{N,k}(t) \Phi(2^{N+D}\xi-k).
\end{align*}
Let $\alpha>n+n/r$ be an integer. Since $\Phi$ is a Schwartz function, we have
$$
|\Phi(\xi)| \lesssim (1+|\xi|)^{-\alpha}.
$$
This yields
\begin{align*}
&|(-\Delta)^{\frac{s}{2}} [\chi_{A_D}(\cdot)m_t(2^D\cdot)](\xi)|\\
&\lesssim\sum_{N=\max(1,2^{D-1})}^{\min(K, 2^{D+1})} \sum_{k\in \N^n:~ N2^N <|k|<(N+1/2)2^N} |\Phi(2^{N+D}\xi-k)|\\
&\lesssim\sum_{N=\max(1,2^{D-1})}^{\min(K, 2^{D+1})} \sum_{k\in \N^n:~ N2^N <|k|<(N+1/2)2^N} (1+|2^{N+D}\xi-k|)^{-\alpha} \\
&\lesssim \sum_{N=\max(1,2^{D-1})}^{\min(K, 2^{D+1})} \int_{\{y\in \R^n:~ N2^N<|y|<(N+1/2)2^N\}} (1+|2^{N+D}\xi-y|)^{-\alpha}\,dy\\
&\approx\sum_{N=\max(1,2^{D-1})}^{\min(K, 2^{D+1})} \int_{\{z\in \R^n:~ N2^N<|z+2^{N+D}\xi|<(N+1/2)2^N\}} (1+|z|)^{-\alpha}\,dz.
\end{align*}

Let us now fix $N=\max(1,2^{D-1}), \dots,\min(K, 2^{D+1})$, and let $z\in \R^n$ satisfy $N2^N<|z+2^{N+D}\xi|<(N+1/2)2^N$. Then
$$
|z| \geq |z+2^{N+D}\xi|-2^{N+D}|\xi|>2^N(N-2^D|\xi|),
$$
and also
$$
|z|\geq 2^{N+D}|\xi| -|z+2^{N+D}\xi|>2^N(2^D|\xi|-N-\frac{1}{2}).
$$
Thus, $|z|>2^N$ if either $N\geq 2^D|\xi|+1$ or $N\leq 2^D|\xi|-3/2$. This means that, for a given $\xi$, inequality $|z|>2^N$ holds for all but three values of $N$. Consequently,
\begin{align}\label{E:boundedness}
&|(-\Delta)^{\frac{s}{2}} [\chi_{A_D}(\cdot)m_t(2^D\cdot)](\xi)|\\ 
\nonumber
&\lesssim 3\int_{\R^n} (1+|z|)^{-\alpha}\,dz
+\sum_{N=\max(1,2^{D-1})}^{\min(K, 2^{D+1})} \int_{\{z\in \R^n:~ |z|>2^N\}} (1+|z|)^{-\alpha}\,dz\\
\nonumber
&\lesssim C+\sum_{N=1}^\infty 2^{N(n-\alpha)} =C(n,\alpha,s,\Psi).
\end{align}

Assume now that $|\xi|\geq 6$. Since $N\leq 2^{D+1} \leq 2/3 \cdot 2^{D-1}|\xi|$, we have
$$
|z|>2^N(2^D|\xi|-N-\frac{1}{2})\geq 2^N(2^D|\xi|-3/2N) \geq 2^{N+D-1}|\xi|.
$$
Thus, recalling that $D\geq -1$, we obtain
\begin{align}\label{E:decay}
&|(-\Delta)^{\frac{s}{2}} [\chi_{A_D}(\cdot)m_t(2^D\cdot)](\xi)|\\
\nonumber
&\lesssim \sum_{N=\max(1,2^{D-1})}^{\min(K, 2^{D+1})} \int_{\{z\in \R^n:~ |z|>2^{N+D-1}|\xi|\}} (1+|z|)^{-\alpha}\,dz\\
\nonumber
&\lesssim \sum_{N=1}^\infty \int_{\{z\in \R^n:~ |z|>2^{N-2}|\xi|\}} (1+|z|)^{-\alpha}\,dz\\
\nonumber
&\lesssim |\xi|^{n-\alpha} \sum_{N=1}^\infty 2^{(N-2)(n-\alpha)} =C(n,\alpha,s,\Psi) |\xi|^{n-\alpha}.
\end{align}
A combination of estimates~\eqref{E:boundedness} and~\eqref{E:decay} yields
\begin{equation}\label{E:derivative_estimate}
\|(-\Delta)^{\frac{s}{2}} [\chi_{A_D}(\xi)m_t(2^D\xi)]\|_{L^r} \leq C(n,r,s,\Psi).
\end{equation}
Finally, combining estimates~\eqref{E:kato-ponce}, \eqref{E:L^r_estimate} and~\eqref{E:derivative_estimate}, we obtain
the desired conclusion.
\end{proof}

\begin{proof}[Proof of Theorem~\ref{T:main_theorem}]
We may assume, without loss of generality, that $p<2$; the result in the case $p>2$ will then follow by duality.

Let $t$, $K$ and $m_t$ be as described at the beginning of this section, and let $\varphi$ be a Schwartz function such that $\varphi(\xi)=1$ if $|\xi|\leq 2$. Define a function $f$ via its Fourier transform by $\wh{f}(\xi)=\varphi(\xi/K)$. Then $\wh{f}(\xi)=1$ if $|\xi|\leq 2K$. It follows from the proof of Lemma~\ref{L:lemma} that $m_t(\xi)$ is supported in the set $|\xi|<K+3/4\leq 2K$. Therefore, we have
$$
m_t(\xi)\wh{f}(\xi)=m_t(\xi),
$$
and so
\begin{align*}
&T_{m_t}f(x)\\
&=\sum_{N=1}^K c_N \sum_{k\in \N^n:~ N2^N<|k|<(N+1/2)2^N} a_{N,k}(t) 2^{-nN} (\mc F^{-1} \Psi)\left(\frac{x}{2^N}\right) e^{2\pi ix\cdot \frac{k}{2^N}}.
\end{align*}
By Fubini's theorem and Khintchine's inequality, we obtain
\begin{align*}
&\int_0^1 \|T_{m_t} f(x)\|_{L^p}^p\,dt
=\int_{\R^n} \int_0^1 |T_{m_t} f(x)|^p\,dt\,dx\\
&\approx \int_{\R^n} \left(\sum_{N=1}^K \sum_{k\in \N^n:~ N2^N<|k|<(N+1/2)2^N} c_N^2 2^{-2nN} \left|(\mc F^{-1} \Psi)\left(\frac{x}{2^N}\right)\right|^2\right)^{\frac{p}{2}}\,dx\\
&\approx \int_{\R^n} \left(\sum_{N=1}^K c_N^2 N^{n-1} 2^{-nN} \left|(\mc F^{-1} \Psi)\left(\frac{x}{2^N}\right)\right|^2\right)^{\frac{p}{2}}\,dx.
\end{align*}
Let $A>0$ be such that $\mc F^{-1} \Psi$ does not vanish in $\{y\in \R^n: A\leq |y|<2A\}$. Then
\begin{align*}
&\int_0^1 \|T_{m_t} f(x)\|_{L^p}^p\,dt\\
&\gtrsim \int_{\R^n} \left(\sum_{N=1}^K c_N^2 N^{n-1} 2^{-nN} \left|(\mc F^{-1} \Psi)\left(\frac{x}{2^N}\right)\right|^2 \chi_{\{x:~ A\leq \frac{|x|}{2^N} <2A\}}(x)\right)^{\frac{p}{2}}\,dx\\
&\approx \sum_{N=1}^K c_N^p N^{\frac{(n-1)p}{2}} 2^{-\frac{nNp}{2}} \int_{\{x:~ A\leq \frac{|x|}{2^N} <2A\}} \left|(\mc F^{-1} \Psi)\left(\frac{x}{2^N}\right)\right|^p\,dx\\
&\approx \sum_{N=1}^K c_N^p N^{\frac{(n-1)p}{2}} 2^{nN(1-\frac{p}{2})} \int_{\{y:~ A\leq |y| <2A\}}\left|(\mc F^{-1} \Psi)\left(y\right)\right|^p\,dy\\
&\approx \sum_{N=1}^K c_N^p N^{\frac{(n-1)p}{2}} 2^{nN(1-\frac{p}{2})}
\approx \sum_{N=1}^K N^{\frac{(n-1)p}{2}-sp} 2^{N(n-\frac{np}{2}-sp)}\\
&= \sum_{N=1}^K N^{np-n-\frac{p}{2}},
\end{align*}
where the last equality follows from~\eqref{E:assumption_s}. We observe that $np-n-p/2>-1$ as 
$$
p>1>\frac{n-1}{n-\frac{1}{2}}.
$$
Thus,
\begin{equation}\label{E:estimate1}
\int_0^1 \|T_{m_t} f(x)\|_{L^p}^p\,dt \gtrsim K^{np-n-\frac{p}{2}+1}.
\end{equation}

Let us now estimate the $L^p$-norm of $f$. Since $f(x)=K^n (\mc F^{-1} \varphi)(Kx)$, we obtain
\begin{align}\label{E:estimate2}
\|f\|_{L^p}^p 
&=K^{np} \int_{\R^n} |(\mc F^{-1} \varphi)(Kx)|^p\,dx\\
&=K^{np-n} \int_{\R^n} |(\mc F^{-1} \varphi)(y)|^p\,dy
\approx K^{np-n}. \nonumber
\end{align}

Assume that inequality~\eqref{E:not_true} is satisfied. Then, applying~\eqref{E:not_true} with $m=m_t$, integrating with respect to $t$ and using Lemma~\ref{L:lemma}, we get
$$
\int_0^1 \|T_{m_t} f(x)\|_{L^p}^p\,dt \leq C\|f\|_{L^p}^p,
$$
which implies, via~\eqref{E:estimate1} and~\eqref{E:estimate2}, that
$$
K^{np-n-\frac{p}{2}+1}\leq C K^{np-n},
$$
or, equivalently,
\begin{equation}\label{E:contradiction}
K^{1-\frac{p}{2}}\leq C.
\end{equation}
As $p<2$, we have $\lim_{K\to \infty} K^{1-\frac{p}{2}}=\infty$, which contradicts~\eqref{E:contradiction}. The proof is complete. 
\end{proof}

\section*{Acknowledgments}

The author is grateful to Andreas Seeger for useful discussions and to Loukas Grafakos for careful reading of this paper and valuable comments.


\end{document}